\documentclass{elsarticle}
\usepackage{amsmath,amsthm}
\usepackage{bbm}
\newtheorem{df}{Definition}[section]
\newtheorem{prop}[df]{Proposition}
\newtheorem{lem}[df]{Lemma}
\newtheorem{thm}[df]{Theorem}
\newtheorem{cor}[df]{Corollary}
\newtheorem{ex}[df]{Example}
\newtheorem{rem}[df]{Remark}

\def\Z{\mathbbm Z}
\def\C{\mathbbm C}
\def\id{\operatorname{id}}
\journal{arXiv.org}
\begin{document}
\begin{frontmatter}
\title{A nice acyclic matching on the nerve of the partition lattice}
\author[ruelle]{Ralf Donau}
\address[ruelle]{Fachbereich Mathematik, Universit\"at Bremen, Bibliothekstra\ss e 1, 28359 Bremen, Germany}
\ead{ruelle@math.uni-bremen.de}
\begin{abstract}
The author has already proven that the space $\Delta(\Pi_n)/G$ is homotopy equivalent to a wedge of spheres of dimension $n-3$ for all natural numbers $n\geq 3$ and all subgroups $G\subset S_1\times S_{n-1}$. We construct an $S_1\times S_{n-1}$-equivariant acyclic matching on $\Delta(\Pi_n)$ together with a description of its critical simplices. This is also a more elementary approach to determining the number of spheres. We also develop new methods for Equivariant Discrete Morse Theory by adapting the Patchwork Theorem and poset maps with small fibers from Discrete Morse Theory.
\end{abstract}
\end{frontmatter}
\section{Introduction}
Let $n\geq 3$ and let $\Pi_n$ denote the poset consisting of all partitions of $[n]:=\{1,\dots,n\}$ ordered by refinement, such that the finer partition is the smaller partition. Let $\overline{\Pi}_n$ denote the poset obtained from $\Pi_n$ by removing both the smallest and greatest element, which are $\{\{1\},\dots,\{n\}\}$ and $\{[n]\}$, respectively. We consider $\overline{\Pi}_n$ as a category, which is acyclic, and define $\Delta(\overline{\Pi}_n)$ to the nerve of the acyclic category $\overline{\Pi}_n$, which is a regular trisp, see \cite[Chapter 10]{buch}. The symmetric group $S_n$ operates on $\Delta(\overline{\Pi}_n)$ in a natural way.

It is well-known that $\Delta(\overline{\Pi}_n)$ is homotopy equivalent to a wedge of spheres of dimension $n-3$. The following two theorems are results concerning the topology of the quotient $\Delta(\overline{\Pi}_n)/G$, where $G$ is a non-trivial subgroup of $S_n$.
\begin{thm}[Kozlov \cite{clmap}]
For any $n\geq 3$, the topological space $\Delta(\overline{\Pi}_n)/S_n$ is contractible.
\end{thm}
We consider the Young subgroup $S_1\times S_{n-1}:=\{\sigma\in S_n\mid\sigma(1)=1\}$.
\begin{thm}[Donau \cite{donau}]
\label{rdthm}
Let $n\geq3$ and $G\subset S_1\times S_{n-1}$ be a subgroup, then the topological space $\Delta(\overline{\Pi}_n)/G$ is homotopy equivalent to a wedge of $k$ spheres of dimension $n-3$, where $k$ is the index of $G$ in $S_1\times S_{n-1}$.
\end{thm}
This leads to a general question of determining the homotopy type of $\Delta(\overline{\Pi}_n)/G$ for an arbitrary subgroup $G\subset S_n$. One might conjecture that $\Delta(\overline{\Pi}_n)/G$ is homotopy equivalent to a wedge of spheres for any $n\geq 3$ and any subgroup $G\subset S_n$. But unfortunately this statement is not true as the following example will show: Let $p\geq 5$ be a prime number and let $C_p$ denote the subgroup of $S_p$ that is generated by the cycle $(1,2,\dots,p)$. Then the fundamental group of $\Delta(\overline{\Pi}_p)/C_p$ is isomorphic to $\Z/p\Z$. In particular $\Delta(\overline{\Pi}_p)/C_p$ cannot be homotopy equivalent to a wedge of spheres. A proof, which uses facts about covering spaces\footnote{See \cite[Chapter 1.3]{hatcher}}, can be found in \cite{torsion}.


In this paper we construct an $(S_1\times S_{n-1})$-equivariant acyclic matching on the face poset ${\cal F}(\Delta(\overline{\Pi}_n))$ of $\Delta(\overline{\Pi}_n)$ for $n\geq 3$ such that we have a description of the critical simplices. This induces an acyclic matching on $\Delta(\overline{\Pi}_n)/G$ for any subgroup $G\subset S_1\times S_{n-1}$.

Equivariant acyclic matchings are also useful to find equivariant homotopies between spaces, since there exists an equivariant version of the Main Theorem of Discrete Morse Theory, see \cite{freij}. For the construction of an equivariant acyclic matching we have similar tools as in Discrete Morse Theory. An equivariant closure operator induces an equivariant trisp closure map which induces an equivariant acyclic matching. A detailed description of the non-equivariant versions of these tools can be found in \cite{clmap}.

Consider the Young subgroup $S_{k_1}\times S_{k_2}\times\dots\times S_{k_r}\subset S_n$ with $k_1+k_2+\dots+k_r=n$. During my work on this paper Gregory Arone\footnote{See Example 1.4 in \cite{arone}} discovered the following: $\Delta(\overline{\Pi}_n)$ is $(S_{k_1}\times S_{k_2}\times\dots\times S_{k_r})$-homotopy equivalent to a wedge of $(n-1)!$ spheres of dimension $n-3$ if $\gcd(k_1,k_2,\dots,k_r)=1$. $S_{k_1}\times S_{k_2}\times\dots\times S_{k_r}$ acts freely on these spheres. This is a generalization of Corollary \ref{moeglichkeit2} and Lemma \ref{freetrans}. Notice that a preprint of this paper was already published on arXiv in 2012.
\section{Discrete Morse Theory}
The definitions of regular trisps, partial matchings, acyclic matchings and foundations of Discrete Morse Theory can be found in \cite{forman,clmap,buch}. The following two theorems of Discrete Morse Theory are frequently used in our proofs.
\begin{thm}[Patchwork Theorem]
\label{patchwork}
Let $\varphi:P\longrightarrow Q$ be an order-preserving map and assume we have acyclic matchings on the subposets $\varphi^{-1}(q)$ for all $q\in Q$. Then the union of these matchings is an acyclic matching on $P$.
\end{thm}
\begin{thm}
\label{morsemain}
Let $\Delta$ be a finite regular trisp and let $M$ be an acyclic matching on the poset ${\cal F}(\Delta)\setminus\{\hat{0}\}$. Let $c_i$ denote the number of critical $i$-dimensional simplices of $\Delta$. Then $\Delta$ is homotopy equivalent to a CW complex with $c_i$ cells of dimension $i$.
\end{thm}
The proofs of Theorems \ref{patchwork} and \ref{morsemain} as well as further facts on Discrete Morse Theory can be found in \cite[Chapter 11]{buch}.
\section{An Equivariant Patchwork Theorem}
We wish to construct an equivariant acyclic matching on a poset by gluing together smaller equivariant acyclic matchings on parts of the poset. This is similar to Theorem \ref{patchwork} with the difference that we also create copies of these matchings in our construction, see Figure \ref{eqex}.
\begin{df}
Let $P$ be a poset and let $G$ be a group acting on $P$. Let $M$ be an acyclic matching on $P$. We call $M$ an \emph{$G$-equivariant acyclic matching} if $(a,b)\in M$ implies $(ga,gb)\in M$ for all $g\in G$ and $a,b\in P$.
\end{df}
Let $G$ be a group acting on some posets $P$ and $Q$. For an element $q\in Q$ we set $G_q:=\{g\in G\mid gq=q\}$, known as the stabilizer subgroup of $q$.
\begin{prop}
\label{eqpatchwork}
Let $\varphi:P\longrightarrow Q$ be an order-preserving $G$-map and let $R\subset Q$ be a subset such that $R$ contains exactly one representative for each orbit in $Q$. Assume for each $r\in R$ we have an $G_r$-equivariant acyclic matching $M_r$ on $\varphi^{-1}(r)$. For $r\in R$, let $C_r$ denote the set of critical elements of $M_r$. Then we have an $G$-equivariant acyclic matching on $P$ such that
\[
\bigcup_{g\in G, r\in R}gC_r
\]
is the set of critical elements.

Let $r\in R$ and assume $G_r$ acts transitively on $C_r$. Then $G$ acts transitively on
\[
\bigcup_{g\in G}gC_r
\]
\end{prop}
\begin{proof}
We define acyclic matchings on the fibers of $\varphi$ as follows. For each $q\in Q$ we choose $r\in R$ and $g\in G$ with $gr=q$. The map $g:\varphi^{-1}(r)\longrightarrow\varphi^{-1}(q)$, which is an isomorphism of posets, induces an acyclic matching on $\varphi^{-1}(q)$. If we choose another $h\in G$ with $hr=q$, then we obtain the same matching. By Theorem \ref{patchwork} the union of these acyclic matchings is an acyclic matching which is $G$-equivariant by construction. The second statement is easy to see.
\end{proof}
\begin{figure}[ht]
\centering
\includegraphics{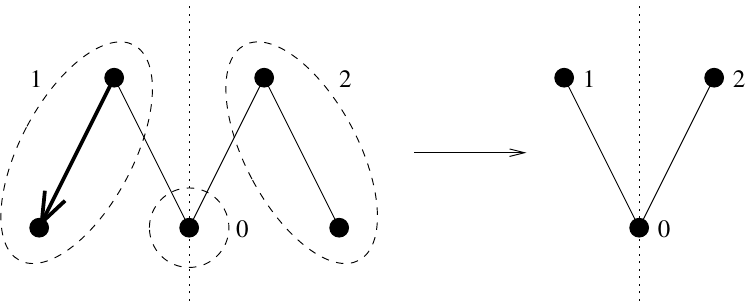}
\caption{A simple example: An $\Z_2$-equivariant acyclic matching composed of acyclic matchings on the fibers of $0$ and $1$. The matching pair in the fiber of $1$ is copied to the fiber of $2$. $\Z_2$ acts on both posets by reflection across the vertical line.}
\label{eqex}
\end{figure}
\begin{rem}
\label{quotmatching}
Let $G$ be a group acting on a finite regular trisp $\Delta$. Assume we have an $G$-equivariant acyclic matching $M$ on ${\cal F}(\Delta)\setminus\{\hat{0}\}$. Let $C$ be the set of critical simplices. Clearly we have an action of $G$ on $C$. Let $H\subset G$ be a subgroup. Then $M/H$ is an acyclic matching on ${\cal F}(\Delta/H)\setminus\{\hat{0}\}$, where $C/H$ is the set of critical simplices. In particular, if $\Delta$ is $G$-collapsible, then $\Delta/H$ is collapsible. Furthermore if $H$ is a normal subgroup, then the acyclic matching $M/H$ is $(G/H)$-equivariant.
\end{rem}
We also have an equivariant version of the Main Theorem of Discrete Morse Theory.
\begin{thm}[Freij \cite{freij}]
\label{freij}
Let $G$ be a finite group. Let $\Delta$ be a finite regular $G$-trisp and let $M$ be a $G$-equivariant acyclic matching on the poset ${\cal F}(\Delta)\setminus\{\hat{0}\}$. Let $c_i$ denote the number of critical $i$-dimensional simplices of $\Delta$. Then $\Delta$ is $G$-homotopy equivalent to a $G$-CW complex where the cells correspond to the critical simplices of $M$ and the action of $G$ is the same as the action on $\Delta$ restricted to the critical simplices of $M$.
\end{thm}
There exists a characterization of acyclic matchings by means of order-preserving maps with small fibers, see Definition 11.3 and Theorem 11.4 in \cite[Chapter 11]{buch}. In a similar way we can also characterize $G$-equivariant acyclic matchings by means of order-preserving $G$-maps with small fibers.

For an order-preserving map with small fibers $\varphi$ let $M(\varphi)$ denote its associated acyclic matching which consists of all fibers of cardinality $2$, see \cite[Chapter 11]{buch}.
\begin{prop}
\label{smallfibers}
Let $G$ be a group acting on a finite poset $P$. For any order-preserving $G$-map $\varphi:P\longrightarrow Q$ with small fibers, the acyclic matching $M(\varphi)$ is $G$-equivariant. On the other hand, any $G$-equivariant acyclic matching $M$ on $P$ can be represented as $M=M(\varphi)$, where $\varphi:P\longrightarrow Q$ is an order-preserving $G$-map with small fibers.
\end{prop}
For two disjoint posets $(P_1,\leq_{P_1})$, $(P_2,\leq_{P_2})$ let $P_1\oplus P_2$ denote the \emph{ordinal sum}\footnote{Called \emph{stack} in \cite{buch}} of $P_1$ and $P_2$, where the set of $P_1\oplus P_2$ is $P_1\cup P_2$ and for two elements $x,y\in P_1\cup P_2$ we set $x\leq y$ if one of the following three conditions is satisfied:
\begin{enumerate}
\item $x,y\in P_1$ and $x\leq_{P_1} y$
\item $x,y\in P_2$ and $x\leq_{P_2} y$
\item $x\in P_1$ and $y\in P_2$
\end{enumerate}
Furthermore if a group $G$ acts on $P_1$ and on $P_2$, then we obtain a canonical action of $G$ on $P_1\oplus P_2$.
\begin{proof}
The first statement is easy to see. Assume we are given an $G$-equivariant acyclic matching $M$ on $P$.

We construct the poset $Q$ with an action of $G$ and the $G$-map $\varphi:P\longrightarrow Q$ inductively by adding either elements of $P$ or elements of $M$ to $Q$ in each step and extending the domain $D\subset P$ of the map $\varphi:D\longrightarrow Q$ such that $D$ has the following property: $d\geq p$ implies $p\in D$ for all $d\in D$ and all $p\in P$. We repeat these steps until we have reached $D=P$. We start with $Q=\emptyset$, $D=\emptyset$ and the trivial map $\varphi:D\longrightarrow Q$. Let $W$ denote the set of minimal elements in $P\setminus D$. We consider two cases.

First we assume one element $c\in W$ is critical. The orbit $Gc$ of $c$ is a subposet of $P$ where each element is only comparable to itself.
\begin{eqnarray*}
\widetilde\varphi:D\cup Gc&\longrightarrow&Q\oplus Gc\\
x&\longmapsto&
\begin{cases}
\varphi(x)&\text{for $x\in D$}\\
x&\text{for $x\in Gc$}\\
\end{cases}
\end{eqnarray*}
Assume $x,y\in D\cup Gc$ with $x\leq y$. $y\in D$ implies $x\in D$. $x\in D$ and $y\in Gc$ implies $\widetilde\varphi(x)\in Q$ and $\widetilde\varphi(y)\in Gc$. Hence $\widetilde\varphi$ is an order-preserving map. Now we assume $d\in D\cup Gc$, $p\in P$ with $d\geq p$. $d\in D$ in implies $p\in D$ by induction hypothesis. For $d\in Gc$ notice that $c$ is a minimal element in $P\setminus D$. Proceed with $Q=Q\oplus Gc$, $D=D\cup Gc$ and $\varphi=\widetilde\varphi$.

Now we assume that all elements in $W$ are matched. For $(a,b)\in M$ we write $a=d(b)$ and $b=u(a)$. We choose an $a\in W$ such that the only element in $W\cup u(W)$ that is smaller that $u(a)$ is $a$ itself. Such an element exists by the proof of Theorem 11.2 in \cite[Chapter 11]{buch}. We consider the orbit $G(a,u(a))$ of $(a,u(a))\in M$ as a poset where each element is only comparable to itself.
\begin{eqnarray*}
\widetilde\varphi:D\cup Ga\cup Gu(a)&\longrightarrow&Q\oplus G(a,u(a))\\
x&\longmapsto&
\begin{cases}
\varphi(x)&\text{for $x\in D$}\\
(x,u(x))&\text{for $x\in Ga$}\\
(d(x),x)&\text{for $x\in Gu(a)$}\\
\end{cases}
\end{eqnarray*}
We notice that $x<y$ implies $(x,y)\in M$ for all $x,y\in Ga\cup Gu(a)$ by the choice of $a$, hence $\widetilde\varphi(x)=\widetilde\varphi(y)$. $\widetilde\varphi$ is order-preserving and $D$ has the desired property by almost the same arguments as in the first case. Proceed with $Q=Q\oplus G(a,u(a))$, $D=D\cup Ga\cup Gu(a)$ and $\varphi=\widetilde\varphi$.

$\varphi$ is an order-preserving $G$-map with small fibers by construction. Clearly we have $M=M(\varphi)$. Notice that $\varphi$ is even surjective.
\end{proof}
\section{The main result}
Let $n\geq 3$ be a fixed natural number.
\begin{df}
Let $A$ be the set of all vertices of $\Delta(\overline{\Pi}_n)$ where all blocks not containing $1$ are singleton. We define the following set of simplices of $\Delta(\overline{\Pi}_n)$, see Figure \ref{exsimplex}.
\[
C_n:=\{\sigma\in{\cal F}(\Delta(\overline{\Pi}_n))\mid\text{$V(\sigma)\subset A$ and $\dim\sigma=n-3$}\}
\]
$V(\sigma)$ denotes the set of vertices of $\sigma$ and ${\cal F}(\Delta(\overline{\Pi}_n))$ denotes the face poset of $\Delta(\overline{\Pi}_n)$. Furthermore we set $\alpha_n$ to the vertex $\{\{1\},\{2,\dots,n\}\}$.
\end{df}
\begin{figure}[ht]
\centering
\includegraphics{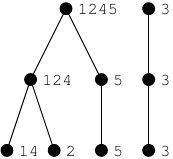}
\caption{A simplex in $C_5$ which has dimension $2$.}
\label{exsimplex}
\end{figure}
\begin{rem}
\label{card}
The cardinality of $C_n$ is $(n-1)!$.
\end{rem}
\begin{prop}
\label{main}
There exists an $(S_1\times S_{n-1})$-equivariant acyclic matching on ${\cal F}(\Delta(\overline{\Pi}_n))$, such that $C_n\cup\{\alpha_n\}$ is the set of critical simplices.
\end{prop}
Let $V$ be the set of all vertices where the block containing $1$ has exactly two elements and any other block is singleton. Such a vertex can be written as
\[
v_k:=\{\{1,k\},\{2\},\dots,\widehat{\{k\}},\dots,\{n\}\}
\]
with $k\in\{2,\ldots, n\}$. The element with the hat above is omitted.

We define a poset $P:=V\cup\{0\}$ such that $0$ is the smallest element of $P$ and the only element that is comparable with some other element. That means $x,y\in P$, $x<y$ implies $x=0$. We define an order-preserving map $\varphi:{\cal F}(\Delta(\overline{\Pi}_n))\longrightarrow P$ as follows. Let $\sigma\in{\cal F}(\Delta(\overline{\Pi}_n))$, then we map $\sigma$ to $0$ if $V(\sigma)\cap V=\emptyset$. Otherwise we map $\sigma$ to the special vertex of $V$ that belongs to $\sigma$, which is unique. Notice that $S_1\times S_{n-1}$ acts on $P$ in a natural way and $\varphi$ is a $(S_1\times S_{n-1})$-map. $P$ has two orbits where one consists of one element which is $0$. The other orbit may be represented by $v_n$.
\begin{lem}
\label{mitte}
There exists an $(S_1\times S_{n-1})$-equivariant acyclic matching on $\varphi^{-1}(0)$, such that $\alpha_n$ is the only critical simplex.
\end{lem}
\begin{proof}
The proof of Lemma \ref{mitte} is the same as the proof of Lemma 3.2 in \cite{donau} for the case $G=\{\id_{[n]}\}$. It is easy to see that the acyclic matching, that is constructed in this proof, is $(S_1\times S_{n-1})$-equivariant.
\end{proof}
\begin{proof}[Proof of Proposition \ref{main}]
It is easy to see, that the statement is true for $n=3$. Now we assume $n>3$ and proceed by induction.

We define acyclic matchings on $\varphi^{-1}(0)$ and $\varphi^{-1}(v_n)$ as follows. By Lemma \ref{mitte} there exists an $(S_1\times S_{n-1})$-equivariant acyclic matching on $\varphi^{-1}(0)$, where $\alpha_n$ is the only critical simplex.

We define a map
\[
\psi:{\cal F}(\Delta(\overline{\Pi}_{n-1}))\longrightarrow\varphi^{-1}(v_n)\setminus\{v_n\}
\]
as follows. We add $n$ to the block that contains $1$ in each partition and append $v_n$ to the bottom of the chain, see Figure \ref{exmapping}. The map $\psi$ is an isomorphism of posets. A similar definition of $\psi$ as well as a detailed description of its inverse can be found in the proof of Lemma 4.1 in \cite{donau}.
\begin{figure}[ht]
\centering
\includegraphics{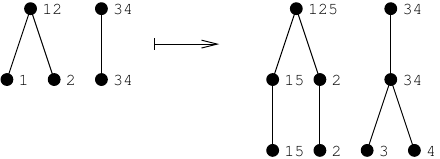}
\caption{Example for the map $\psi$, where $n=5$.} 
\label{exmapping}
\end{figure}

Via $\psi$ we get an acyclic matching $M$ on $\varphi^{-1}(v_n)$, where the set of critical simplices consists of the simplices in $\psi[C_{n-1}]$, one critical simplex $s_n$ consisting of the two vertices $v_n$ and $\{\{1,n\},\{2,\dots,(n-1)\}\}$, which has dimension $1$. Additionally we have the critical simplex that has only the vertex $v_n$. Finally we match $v_n$ with $s_n$.

We have to show that $\sigma(v_n)=v_n$ and $(a,b)\in M$ implies $(\sigma a,\sigma b)\in M$ for all $\sigma\in S_1\times S_{n-1}$ and all $a,b\in\varphi^{-1}(v_n)$. Let $\sigma\in S_1\times S_{n-1}$. $\sigma(v_n)=v_n$ implies $\sigma(1)=1$ and $\sigma(n)=n$. We define a $\widetilde\sigma\in S_1\times S_{n-2}$ by setting $\widetilde\sigma(x):=\sigma(x)$ for $1\leq x\leq n-1$. Notice $\sigma\psi=\psi\widetilde\sigma$ which implies $\psi^{-1}\sigma=\widetilde\sigma\psi^{-1}$. Let $(a,b)\in M$. Clearly we have $(v_n,s_n)=(\sigma v_n,\sigma s_n)$. Now we assume $a\not=v_n$ and $b\not=s_n$. By induction hypothesis, we have an acyclic matching $\widetilde M$ on ${\cal F}(\Delta(\overline{\Pi}_{n-1}))$ which is $(S_1\times S_{n-2})$-equivariant. By the construction of $M$ we have $(\psi(a)^{-1},\psi(b)^{-1})\in\widetilde M$. This implies $(\widetilde\sigma\psi(a)^{-1},\widetilde\sigma\psi(b)^{-1})\in\widetilde M$, hence $(\sigma a,\sigma b)\in M$.

By Proposition \ref{eqpatchwork} there exists an $(S_1\times S_{n-1})$-equivariant acyclic matching on ${\cal F}(\Delta(\overline{\Pi}_n))$ such that
\[
\left(\bigcup_{g\in S_1\times S_{n-1}}g\psi[C_{n-1}]\right)\cup\{\alpha_n\}
\]
is the set of critical elements. It is easy to see that this set equals $C_n\cup\{\alpha_n\}$.
\end{proof}
\begin{cor}
\label{moeglichkeit1}
Let $G\subset S_1\times S_{n-1}$ be a subgroup. Then there exists an acyclic matching on ${\cal F}(\Delta(\overline{\Pi}_n))/G$, such that the set of critical simplices consists of the simplices in $C_n/G$ and $\alpha_n/G$.
\end{cor}
\begin{proof}
Apply Remark \ref{quotmatching}.
\end{proof}
\begin{ex}
\label{nicematchingex}
Assume $G=S_1\times S_{n-1}$. The vertices of $\Delta(\overline{\Pi}_n)/S_1\times S_{n-1}$ can be indexed with number partitions of $n$, which we may write as $v_0\oplus v_1+\dots+v_r$, that distinguish the first number, i.e.\ $\oplus$ is non-commutative. The number on the left side of $\oplus$, that is $v_0$, corresponds to the block that contains $1$. There exists an acyclic matching on the poset ${\cal F}(\Delta(\overline{\Pi}_n)/S_1\times S_{n-1})$, where the set of critical simplices consists of the vertex $1\oplus(n-1)$ and the simplex $\sigma$ of dimension $n-3$ whose vertices are $v_0\oplus 1^{n-v_0}$ with $v_0=2,\dots,n-1$.
\end{ex}
A slightly different proof of the result in Example \ref{nicematchingex}, as well as a detailed description of $\Delta(\overline{\Pi}_n)/S_1\times S_{n-1}$, can be found in \cite{donau2}.
\section{Applications}
Let $n\geq 3$ be a fixed natural number.
\begin{cor}
\label{moeglichkeit2}
The topological space $\Delta(\overline{\Pi}_n)$ is $(S_1\times S_{n-1})$-homotopy equivalent to a wedge of $(n-1)!$ spheres of dimension $n-3$. The spheres are indexed with the simplices in $C_n$, which induces an action of $S_1\times S_{n-1}$ on the $(n-1)!$ spheres.
\end{cor}
\begin{proof}
Apply Theorem \ref{freij}.
\end{proof}
\begin{lem}
\label{freetrans}
$S_1\times S_{n-1}$ acts freely and transitively on $C_n$.
\end{lem}
\begin{proof}
Since $C_n=\bigcup_{g\in S_1\times S_{n-1}}g\psi[C_{n-1}]$, the action is transitive, which follows inductively by the second statement of Proposition \ref{eqpatchwork}.

By Remark \ref{card} the cardinality of $C_n$ is $(n-1)!$ which equals the cardinality of $S_1\times S_{n-1}$. Hence the action is free.
\end{proof}
Let $G\subset S_1\times S_{n-1}$ be an arbitrary subgroup.
\begin{rem}
The cardinality of $C_n/G$ is the index of $G$ in $S_1\times S_{n-1}$.
\end{rem}
\begin{proof}
Apply Lemma \ref{freetrans}.
\end{proof}
Now Theorem \ref{rdthm} follows as a corollary. We can either apply Corollary \ref{moeglichkeit1} or Corollary \ref{moeglichkeit2}.
\begin{cor}
The topological space $\Delta(\overline{\Pi}_n)/G$ is homotopy equivalent to a wedge of spheres of dimension $n-3$. The number of spheres is the index of $G$ in $S_1\times S_{n-1}$.
\end{cor}
Now we consider the top cohomology $H^{n-3}(\Delta(\overline{\Pi}_n);\C)$ of $\Delta(\overline{\Pi}_n)$. For each simplex $\sigma\in C_n$ we obtain a $\sigma^*\in H^{n-3}(\Delta(\overline{\Pi}_n);\C)$ which is represented by a cocycle that maps $\sigma$ to $1$ and all other simplices to $0$. We obtain a basis $\{\sigma^*\mid\sigma\in C_n\}$ of the top cohomology which is called \emph{splitting basis} in \cite{wachs}.
\section*{Acknowledgments}
The author would like to thank Dmitry N. Kozlov for this interesting problem, Ragnar Freij and Giacomo d'Antonio for the helpful discussions. Furthermore the author would like to thank another anonymous reader of a previous version of this paper.
\bibliographystyle{model1-num-names}

\end{document}